\documentclass[12pt]{article}

\usepackage{amssymb,amsthm,amsmath}
\usepackage[usenames,dvipsnames]{color}
\usepackage{soul}
\usepackage{longtable}
\usepackage{siunitx}
\allowdisplaybreaks

\newcommand{\comment}[1]{}
\definecolor{teal}{RGB}{0,128,128}
\definecolor{darkpurple}{RGB}{128,0,128}

\usepackage[titletoc,toc,title]{appendix}

\newcommand{\tmod}[1]{{\;\rm (mod\; #1)}}

\newtheorem{theorem}{Theorem}[section]

\newtheorem{lemma}[theorem]{Lemma}
\newtheorem{cor}[theorem]{Corollary}

\theoremstyle{definition}
\newtheorem{defn}[theorem]{Definition}

\def \cH {{\cal H}}

\def \Z {\mathbb Z}

\title {On the Hamilton-Waterloo Problem\\ with cycle lengths of distinct parities}

\author{
A.\ C.\ Burgess \footnotemark[1]  
\and 
P.\ Danziger    \footnotemark[2] 
\and
T.\ Traetta     \footnotemark[3] 
}

\begin{document}

\maketitle

\footnotetext[1]{Department of Mathematics and Statistics, University of New Brunswick, 100 Tucker Park Rd., Saint John, NB  E2L 4L5, Canada. Email: andrea.burgess@unb.ca.}
\footnotetext[2]{Department of Mathematics, Ryerson University, 350 Victoria St., Toronto, ON  M5B 2K3, Canada. Email: danziger@ryerson.ca.}
\footnotetext[3]{Department of Mathematics, Ryerson University, 350 Victoria St., Toronto, ON  M5B 2K3, Canada. Email: tommaso.traetta@ryerson.ca, traetta.tommaso@gmail.com.}

\begin{abstract}
Let $K_v^*$ denote the complete graph $K_v$ if $v$ is odd and $K_v-I$, the complete graph with the edges of a 1-factor removed, if $v$ is even.
Given non-negative integers $v, M, N, \alpha, \beta$, the Hamilton-Waterloo problem asks for a $2$-factorization of  $K^*_v$  into $\alpha$  $C_M$-factors 
and $\beta$ $C_N$-factors.
Clearly, $M,N\geq 3$, $M\mid v$, $N\mid v$ and $\alpha+\beta = \lfloor\frac{v-1}{2}\rfloor$ are necessary conditions. 

Very little is known on the case where $M$ and $N$ have different parities.
In this paper, we make some progress on this case by showing, among other things,
that  the above necessary conditions are sufficient 
whenever $M|N$, $v>6N>36M$, and $\beta\geq 3$. 

%
\end{abstract}


\section{Introduction}
As usual, we denote by $V(G)$ and $E(G)$ the {\em vertex set} and {\em the edge set} of 
a simple graph $G$, respectively. Also, we denote by $tG$ the vertex-disjoint union of $t>0$ copies of $G$.

A {\em factor} of $G$ is a spanning subgraph of $G$; in particular, a 1-factor is a factor which is 1-regular and a 2-factor is a factor which is 2-regular and hence consists of a collection of cycles. A $2$-factor of $G$ containing only one cycle is a 
{\em Hamiltonian cycle}.
We denote by $C_\ell$ a cycle of length $\ell$ (briefly, an $\ell$-cycle), by
$(x_0,x_1\ldots, x_{\ell-1})$ the $\ell$-cycle with edges $x_0x_1, x_1x_2, \ldots, x_{\ell-1} x_0$, and by $K_v$ the {\em complete graph} on $v$ vertices. By $K_v^*$ we mean the graph $K_v$ when $v$ is odd and $K_v-I$, where $I$ is a single 1-factor, when $v$ is even.

A 2-factorization of a simple graph $G$ is a set of $2$-factors of $G$ whose edge sets
partition $E(G)$. It is well known that a regular graph has a 2-factorization if and only if every vertex has even degree.
However, if we specify a particular 2-factor, $F$ say, and ask for all the factors to be isomorphic to $F$ the problem becomes much harder. Indeed, if $G\cong K_v^*$, we have the Oberwolfach Problem, which is well known to be hard. A survey of the well-known results on this problem, updated to 2006, can be found in 
\cite[Section VI.12]{Handbook}. For more recent results we refer the reader to 
\cite{Bryant Schar 09,  BryantDanzigerDean, BryantDanzigerPettersson, RiTr11, Traetta 13}.

Given a simple graph $G$ and a collection of graphs $\cH$, an {\em $\cH$-factor of $G$} is a set of vertex-disjoint subgraphs of $G$, each isomorphic to a member of $\cH$, which between them cover every point in $G$. An {\em $\cH$-factorization of $G$} is a set of edge-disjoint $\cH$-factors of $G$ whose edges partition the edge set of $G$. When $\cH$ consists of a single graph $H$, we speak of $H$-factors and $H$-factorizations of $G$ respectively. If $H$ is a Hamiltonian cycle of $G$ and there exists an 
$H$-factorization of $G$ (briefly, a {\em Hamiltonian factorization}), then 
$G$ is called {\em Hamiltonian factorable}. 

We call a factor whose components are pairwise isomorphic a {\em uniform} factor. 
The problem of factoring $K_v^*$ into pairwise isomorphic uniform 2-factors has
been solved \cite{Alspach Haggvist 85, ASSW, Hoffman Schellenberg 91}.
\begin{theorem}[\cite{Alspach Haggvist 85, ASSW, Hoffman Schellenberg 91}]
\label{Hamiltonian}
\label{OP uniform}
Let $v, \ell \geq 3$ be integers.
There is a $C_\ell$-factorization of $K_v^*$ if and only if $\ell \mid v$, except that there is no $C_3$-factorization of $K_{6}^*$ or $K_{12}^*$.
\end{theorem}

Given a graph $G$, we denote by $G[n]$ the {\em lexicographic product} of $G$ with the empty graph on $n$ points. Specifically, the vertex set of $G[n]$ is $V(G)\times \Z_n$ (where $\Z_n$ denotes the cyclic group of order $n$) and 
$(x,i)(y,j)\in E(G[n])$ if and only if $xy \in E(G)$, $i,j\in \Z_n$.
Note that $G[n_1][n_2] \cong G[n_1n_2]$. 
%

The existence problem for a  
$C_\ell$-factorization of the complete equipartite graph has been completely solved
by Liu \cite{Liu00, Liu03}.
\begin{theorem}[\cite{Liu00, Liu03}]
\label{Liu} Let $\ell, t$ and $z$ be positive integers with $\ell\geq 3$.
There exists a $C_\ell$-factorization of $K_t[z]$ if and only if $\ell\mid tz$, $(t-1)z$ is even, 
further $\ell$ is even when $t = 2$, and
$(\ell, t, z) \not\in \{(3, 3, 2), (3, 6, 2), (3, 3, 6), (6, 2, 6)\}$.
\end{theorem}

We provide a straightforward generalization of Theorem \ref{Liu} to
$C_\ell[n]$-factorizations of $K_t[zn]$.

\begin{cor}
\label{LiuGen}
Given four positive integers $\ell, n, t$ and $z$ with $\ell\geq 3$,
there exists a $C_\ell[n]$-factorization of $K_t[z][n]\cong K_t[nz]$ whenever 
$\ell\mid tz$, $(t-1)z$ is even, $\ell$ is even when $t = 2$, and
$(\ell, t, z) \not\in \{(3, 3, 2)$, $(3, 6, 2), (3, 3, 6), (6, 2, 6)\}$.
\end{cor}
\begin{proof}
  Theorem \ref{Liu} guarantees the existence of a
  $C_\ell$-factorization of $K_t[z]$. By expanding each point of this factorization by 
  $N$, we obtain a $C_\ell[n]$-factorization of $K_t[z][n] \cong K_t[nz]$.
\end{proof}

A well-known variant of the Oberwolfach Problem is the Hamilton-Wa\-ter\-loo Problem 
HWP$(G;F,F';\alpha,\beta)$, which
asks for a factorization of a specified graph $G$ into 
$\alpha$ copies of $F$ and $\beta$ copies of  $F'$,
where $F$ and $F'$ are distinct $2$-factors of $G$.
We denote by HWP$(G; F, F')$ the set of $(\alpha, \beta)$ for which there is a solution 
to HWP$(G;F,F';\alpha,\beta)$.
In the case where $F$ and $F'$ are uniform with cycle lengths $M$ and $N$, respectively, we refer to 
HWP$(G;M,N;\alpha,\beta)$ and  HWP$(G; M, N)$ as appropriate. 
Further, if $G=K^*_v$, we refer to HWP$(v;M,N;\alpha,\beta)$ and HWP$(v; M, N)$ respectively. 
We note the following necessary conditions for the case of uniform factors.

\begin{theorem}
\label{nec} Let $G$ be a graph of order $v$, and let 
$M,N, \alpha$ and $\beta$ be  non-negative integers.
In order for a solution of HWP$(G; M,N; \alpha,\beta)$ to exist,
$M$ and $N$ must be divisors of $v$ greater than $2$, and $G$ must be regular of degree $2(\alpha+\beta)$. 
\end{theorem}
This problem has received much interest recently.
For more details and some history on the problem, we refer the reader to \cite{BDT1, BDT2}. 
These two papers deal with the case where both $M$ and $N$ are odd positive integers
and provide an almost complete solution to the Hamilton-Waterloo Problem HWP$(v;M,N;\alpha,\beta)$ for odd $v$.
If $M$ and $N$ are both even, then HWP$(v;M,N;\alpha,\beta)$ has a solution except possibly when $\alpha=1$ or 
$\beta=1$ \cite{BryantDanziger}, whereas this problem is completely solved when $M$ and $N$ are even and $M$ is a divisor of $N$ \cite{BryantDanzigerDean}.

In this paper, we deal with the challenging case where $M$ and $N$ have different parities.
In fact, the only known results on ${\rm HWP}(v; M,N; \alpha, \beta)$ when $M\not\equiv N \tmod{2}$ concern
the case $(M,N)=(3,4)$ which is completely solved in 
\cite{BonviciniBuratti,  DanzigerQuattrocchiStevens, OdabasiOzkan, WangChenCao},
and the cases where $(M,N)=(3,v)$ \cite{LeiShen}, $(M,N)=(3,6n)$ \cite{AsplundEtAl} 
or $(M,N)=(4,N)$ \cite{KeranenOzkan, OdabasiOzkan} which are all still open.

In this paper, we make further progress by showing the following.

\begin{theorem}\label{main}
  Let $M,N,v, \alpha,\beta$ be positive integers such that $N>M\geq 3$ and $M$ is an odd divisor of $N$.
  Then, $(\alpha, \beta)\in \mathrm{HWP}(v; M, N)$ if and only if 
  $N\mid v$ and $\alpha+\beta=\lfloor\frac{v-1}{2}\rfloor$ except possibly when at least one of the following conditions holds:
  \begin{enumerate}
  \item $\beta=1$;
  \item $\beta=2$, $N\equiv 2M \tmod{4M}$;
  \item $N\in\{2M, 6M\}$;
  \item $v\in\{N, 2N, 4N\}$;
  \item $(M,v)=(3,6N)$. 
  \end{enumerate}
\end{theorem}
In the next section we introduce some tools and provide some powerful methods which we use in Section~\ref{Section C_M[n]} where we prove a result (Theorem~\ref{C_M[n]}) on factorizations of $C_M[n]$, the lexicographic product of an $M$-cycle and the empty graph on $n$ vertices. In Section \ref{Main section}, we prove the main result of this paper, Theorem  \ref{main}.

\section{Preliminaries}
\label{Section prelim}

In this section we state some known results and develop the tools we will need for the 2-factorizations.
We use $[a,b]$ to denote the set of integers from $a$ to $b$ inclusive; clearly, $[a,b]$ is empty when $a>b$. 
  
\subsection{Cayley graphs}
We will make use of the notion of a Cayley graph on an additive group $\Gamma$. 
Given $S\subseteq \Gamma\setminus\{0\}$, the {\em Cayley Graph} cay$(\Gamma, S)$ is a graph with vertex set $\Gamma$ and edge set $\{a(d+a) \mid a\in \Gamma,d\in S\}$.
When $\Gamma = \Z_N$ this graph is known as a {\em circulant graph}.
We note that the edges generated by $d\in S$ are the same as those generated by $-d\in -S$, so that cay$(\Gamma,S) = {\rm cay}(\Gamma, \pm S)$, and that the degree of each point is $|S\cup (-S)|$. 

Given a set $S\subseteq\Gamma$, we denote by
$C_m[S]$ ($m\geq 3$) the graph with point set 
$\mathbb{Z}_m\times \Gamma$ and edges $(i,x)(i+1,d+x)$, $i\in \mathbb{Z}_m$, $x\in\Gamma$ and $d\in S$. In other words, 
$C_m[S]= {\rm cay}(\Z_m\times \Gamma, \{1\}\times S)$; hence, it is $2|S|$-regular.
It is straightforward to see that if $\Gamma$ has order $n$, then $C_{m}[n] \cong C_{m}[\Gamma]$;  
hence, $C_{m}[S]$ is a subgraph of $C_m[n]$.
We will sometimes denote the vertex $(i,x)$ of $C_m[S]$ by $i_x$.

We will make use of the following two results due to Bermond, Favaron and Mah\'{e}o~\cite{Bermond} and Westlund \cite{We14}, which provide sufficient conditions for the existence of a Hamiltonian factorization of a connected Cayley graph of degree 4 and 6.
\begin{theorem}[\cite{Bermond}]
\label{4reg}
Any connected 4-regular Cayley graph on a finite Abelian group has a Hamiltonian factorization.
\end{theorem}

\begin{theorem}[\cite{We14}]
\label{6reg}
If $X = \mathrm{cay}(A, \{e_1, e_2, e_3\})$ is a 6-regular Cayley graph, $A$ is an abelian group of even order 
generated by both $\{e_1,e_2\}$ and $\{e_2, e_3\}$, and $e_2$ has index at least four in $A$, then $X$
has a Hamiltonian factorization.
\end{theorem}

We use these two results to show the existence of a hamiltonian factorization of a 
special connected $6$-regular subgraph of $C_M[n]$.
\begin{lemma}\label{6 regular}
  Let $n\geq 4$ be even and let $M\geq 3$ be such that $Mn\equiv 0\tmod{4}$. 
  Then, $C_M[\{\frac{n}{2}-1, \frac{n}{2}, \frac{n}{2}+1\}]$
  factorizes into three $C_{Mn}$-factors.
\end{lemma}
\begin{proof}
  We recall that $C_M[\{\frac{n}{2}-1, \frac{n}{2}, \frac{n}{2}+1\}]=cay(\Z_M\times\Z_n, \{e_1,e_2,e_3\})$ where $(e_1, e_2, e_3) = ((1,\frac{n}{2}-1), (1,\frac{n}{2}), (1,\frac{n}{2}+1))$.
  
  We first note that for any $x\in \Z_n$ the set $\{(1,x), (1,x+1)\}$ is a system of generators of 
  $\Z_M\times\Z_n$. In fact, $(0,1)=(1,x+1)-(1,x)$ and $(1,0)=(x+1)(1,x)-x(1,x+1)$; therefore, any element
  of $\Z_M\times\Z_n$ is a linear combination of $\{(1,x), (1,x+1)\}$. It then follows that
  both $\{e_1, e_2\}$ and $\{e_2, e_3\}$ generate $\Z_M\times\Z_n$, hence 
  $C_M[\{\frac{n}{2}-1, \frac{n}{2}, \frac{n}{2}+1\}]$ is a connected $6$-regular graph.   
  
  We denote by $\langle e_2\rangle$ the subgroup of $\Z_M\times\Z_n$ generated by $e_2$, and
  by 
  $|\Z_M\times\Z_n:\langle e_2\rangle|$ the index of $\langle e_2\rangle$ in 
  $\Z_M\times\Z_n$. 
  It is not difficult to check that 
  $|\Z_M\times\Z_n:\langle e_2\rangle|= n$ or $\frac{n}{2}$ according to whether $M$ is even or odd. 
  Since by assumption $Mn\equiv 0\tmod{4}$ and $n\geq 4$, we have that 
  $|\Z_M\times\Z_n:\langle e_2\rangle|\geq 4$ when either $M$ is even or $M$ is odd and $n\neq 4$;
  in these cases,
  the assertion follows from Theorem \ref{6reg}. If $M$ is odd and $n=4$, then
  $C_M[\{1,2,3\}]$ can be decomposed into
  $C_M[\{1\}]$, which is a Hamiltonian cycle, and $C_M[\{2,3\}]$ which is a connected $4$-regular Cayley graph and, by Theorem
  \ref{4reg}, it has a Hamiltonian factorization, and this completes the proof.
\end{proof}

\subsection{Constructing factors of $C_M[n]$}

In Section 3 we will make use of the following result which
provide sufficient conditions for the existence of a solution to HWP($C_{\ell}[T]; g\ell, h\ell; \alpha, |T|-\alpha$), where $T$ is a subset of $\Gamma=\Z_{n}$ and $g, h$ are positive divisors of $n$. This result 
 is proven in \cite{BDT2} for an arbitrary group $\Gamma$.

\begin{theorem}[Theorem 2.9, \cite{BDT2}]
\label{matrix}
Let $n$ be a positive integer, and let $g$ and $g'$ be positive divisors of $n$. Also, 
let $T$ be a subset of $\Z_n$ and  $\ell\geq3$. Suppose there exists a  
$|T| \times \ell$ matrix $A=[a_{ij}]$ with entries from $T$ satisfying the following properties:
\begin{enumerate}
    \item 
    \label{constructionD cond 1}
    $\alpha$ rows of $A$ have sum an element of order $g$ in $\Z_n$, and the remaining $|T|-\alpha$ rows
    have sum an element of order $g'$ in $\Z_n$;
    \item
    \label{constructionD cond 2}
	 each column of $A$ is a permutation of $T$.
\end{enumerate}
Then $(\alpha, |T|-\alpha)\in$ HWP($C_{\ell}[T]; g\ell, g'\ell)$. 
Moreover, if we also have that:
\begin{enumerate}
\setcounter{enumi}{2}
    \item 
    \label{constructionD cond 3}
	$T$ is closed under taking negatives,
\end{enumerate}  
  then $(\alpha, |T|-\alpha)\in$ HWP($C_{m}[T]; g m, g' m)$
 for any $m \geq \ell$ with $m\equiv \ell \pmod{2}$.
\end{theorem}
Note that Theorem \ref{matrix} gives 
a $C_{g\ell}$-factorization of $C_{\ell}[T]$ when $\alpha=|T|$.

We finally state the following well-known result which has been proven in 
\cite{Hagg85} in a much more general form.

\begin{lemma}\label{hagg}
$C_M[2]$ has a Hamiltonian factorization for every $M\geq 3$.
\end{lemma}

\subsection{Skolem sequences}

In some of our constructions in Section 3 we will make use of \emph{Skolem sequences}, 
which we now define in a slightly more general form.



\begin{defn}[Skolem sequences]
A {\it Skolem sequence} of {\it order $\nu\geq0$} is a sequence of $\nu+1$ pairs
$(a_0,b_0), (a_1, b_1), \ldots, (a_\nu, b_\nu)$ such that
\begin{enumerate}
    \item $b_i-a_i=i$ for every $i\in[0, \nu]$;
    \item $\displaystyle\bigcup_{i=1}^{\nu} \{a_i,b_i\} = [x,x+2\nu]$ for some integer $x$.
\end{enumerate}
In this case, we say that the Skolem sequence covers the interval $[x,x+2\nu]$.
\end{defn}
We point out that in the literature, the term Skolem sequence is only used when $(x,a_0)=(1,2\nu+1)$.
When $(x,a_0)=(1,2\nu)$, such a sequence is usually referred to as a hooked Skolem sequence. 
In all other cases in which $x=1$, one speaks of an $a_0$-extended Skolem sequence.

We recall the following existence results for Skolem sequences.

\begin{theorem}[\cite{Baker}]
\label{skolem} There exists a Skolem sequence of order $\nu$ for every $\nu\geq 0$
\end{theorem}

Note that given a Skolem sequence $(a_0,b_0), (a_1, b_1), \ldots, (a_\nu, b_\nu)$ covering the interval
$[x, x+2\nu]$ and an integer $t$, it is clear that $(a_0+t,b_0+t), (a_1+t, b_1+t), \ldots, (a_\nu+t, b_\nu+t)$ is still
a Skolem sequence which covers the interval $[x+t, x+2\nu+t]$. 
Therefore, the above theorem implies what follows. 

\begin{cor}\label{skolem cor}
Every interval of length $2\nu+1$ can be covered by a Skolem sequence.
\end{cor}

\section{Determining HWP($C_M[n]; M, Mn$)}
\label{Section C_M[n]}
In this section, we provide sufficient conditions for a solution of 
HWP($C_M[n];$ $M, Mn$) to exists. 
We will make use of Theorem \ref{matrix} to factorize large subgraphs of $C_M[n]$ by constructing suitable matrices with entries in $\Z_n$, and use Theorems
\ref{4reg} and \ref{6reg} to factorize what is possibly left over. 
For this reason, given any integers $x$ and $y$ such that $0<\ell=y-x<n$, we 
define  two $(\ell+1) \times 2$ matrices below, denoted by $A(x,y)$ and $B(x,y)$,  with entries in $\Z_n$:
\[
\scriptsize
\begin{tabular}{c|c|c}
$A(x,y)$  & $B(x,y)$ if $\ell$ is odd & $B(x,y)$ if $\ell$ is even \\ 
 \hline
 & & \\
$
\left[
  \begin{array}{rr}
    x        & -x  \\
    x+1      & -(x+1)      \\  
    \vdots   & \vdots  \\  
    x+\ell   & -(x+\ell) 
  \end{array}
\right] 
$
&
$
\left[
  \begin{array}{rr}
    x        & -(x+1)  \\
    x+1      & -x      \\  
    \vdots   & \vdots  \\  
    x+2i     & -(x+2i+1)  \\
    x+2i+1   & -(x+2i)  \\    
    \vdots   &  \vdots \\
    x+\ell-1 & -(x+\ell)\\
    x+\ell   & -(x+\ell-1) 
  \end{array}
\right] 
$ 
&
$
\left[
  \begin{array}{rr}
    x        & -(x+1)     \\
    x+1      & -x         \\    
    \vdots   & \vdots     \\  
    x+2i     & -(x+2i+1)  \\
    x+2i+1   & -(x+2i)    \\    
    \vdots   &  \vdots    \\
    x+\ell-4 & -(x+\ell-3)\\
    x+\ell-3 & -(x+\ell-4)  \\     
    x+\ell-2 & -(x+\ell-1)\\
    x+\ell-1 & -(x+\ell)  \\ 
    x+\ell   & -(x+\ell-2)\\       
  \end{array}
\right] 
$ 
\end{tabular}
\]
Further, if $y<x$, we set $A(x,y)=\emptyset=B(x,y)$. Finally, $A(x,x)=[x\;\;\;-x]$.
Note that $B(x,y)$ is not defined when $y=x$.

We note that when $x\leq y$ each of the rows in $A(x,y)$ sums to $0$. Similarly, when $x<y$ each of the rows in $B(x,y)$ sums to $\pm1$, unless $y-x$ is even, in which case the last row of $B(x,y)$ sums to $2$. 

We first consider the problem in which $n$ is odd. 
\begin{lemma} \label{C_M[n] n odd}
Let $M,n\geq 3$ with $n$ odd, and let $0\leq\beta\leq n$.
Then $(\alpha,\beta) \in \mathrm{HWP}(C_M[n];$ $M,Mn)$ except possibly when $\beta=1$.
\end{lemma}
\begin{proof}
 Let $T$ be the $n\times 2$ matrix defined as 
$T=
\left[
  \begin{array}{c}
    A(1,\alpha)  \\
    B(\alpha+1, n)   
  \end{array}
\right] 
$.  
Also, let $T'$ be the $n\times3$ matrix obtained from $T$ by replacing each row $[m_1, m_2]$ with
$[\frac{m_1}{2}, \frac{m_1}{2}, m_2]$. 
Here $\frac{m_i}{2}$ is well defined as an element of $\Z_n$, since $n$ is odd.

Clearly, each of the first $\alpha$ rows of $T$ sums to $0$, whereas each of the remaining $\beta$ rows sums to $\pm1$ or $\pm2$ (which are elements of order $n$ in $\Z_n$ since $n$ is odd).
Further, each column of $T$ and $T'$ is a permutation of $Z_n$. 
Therefore, by applying Theorem \ref{matrix} to $T$ and $T'$, it follows that 
$(\alpha,\beta)\in \mathrm{HWP}(C_M[n]; M, Mn)$ for any $M\geq 3$.
\end{proof}

Note that the above Lemma has been independently proven in \cite{KeranenPastine} with different techniques. 
An alternative proof in the case where $M$ is odd can be found in \cite{BDT2}.

The following three lemmas deal with the case where $n$ is even.

\begin{lemma}\label{C_M[n] n even beta=0} If $n\geq 2$ is even and $M\geq3$, then
   $(n,0) \in \mathrm{HWP}(C_M[n];$ $M,Mn)$ except when $M$ is odd and 
   $n=2$ and possibly when $M$ is odd and $n=6$.
\end{lemma}
\begin{proof} We first consider the case where $M\geq3$ is odd. It is not difficult to check that there is no $C_M$-factorization of $C_M[2]$. Therefore, let $n\geq4$ be even with $n\neq 6$. 
By Theorem \ref{Liu} there exists a $C_3$-factorization
$\mathcal{F}=\{F_1, F_2, \ldots, F_n\}$ of $C_3[n]$, where $F_i=\{C_{ij}\mid j\in[1,n]\}$
and $C_{ij}=(c_{ij}^0, c_{ij}^1, c_{ij}^2)$. Without loss of generality we can assume
$c_{ij}^2=(2,j)$ for any $j\in[1,n]$.

Now, for each $i,j\in[1,n]$ we define the $M$-cycle 
$\overline{C}_{ij}=(\overline{c}_{ij}^0, \overline{c}_{ij}^1, \ldots, \overline{c}_{ij}^{M-1})$ 
as follows:
\[ \overline{c}_{ij}^h=
   \begin{cases}
     c_{ij}^h  & \text{if $h=0,1,2$,} \\
     (h,j+i)   & \text{if $h$ is odd  and $3\leq h<M$}, \\
     (h,j)     & \text{if $h$ is even and $4\leq h<M$}. \\     
   \end{cases}
\]
Finally, set $\overline{F}_i=\{\overline{C}_{ij}\mid j\in[1,n]\}$ and 
$\overline{\mathcal{F}}=\{\overline{F}_i\mid i\in[1,n]\}$. It is not difficult to check that each $F_i$ is a $C_M$-factor of $C_M[n]$ and $\overline{\mathcal{F}}$ is a $C_M$-factorization of $C_M[n]$.

If $M\geq 4$ is even, it is enough to apply Theorem \ref{matrix} to the $n\times M$ block matrix 
$T=[A(1,n)\;\;A(1,n)\;\cdots\;A(1,n)]$.
\end{proof}

Note that a result similar to Lemma \ref{C_M[n] n even beta=0} has been proven in \cite{KeranenPastine}
in the case where $M\geq 3$ is odd and $n>1$.

\begin{lemma} \label{C_M[n] n even beta=Mn/2}
Let $n\geq 2$ be even, $M\geq3$, and $0< \beta\leq n$.
Then $(n-\beta,\beta) \in \mathrm{HWP}(C_M[n];$ $M,Mn)$ whenever the following conditions simultaneously hold:
\begin{enumerate}
\item $\beta\equiv \frac{Mn}{2} \tmod{2}$;
\item if $Mn\equiv 2 \tmod{4}$ and $n>2$, then $\beta\neq1$.
\end{enumerate}
\end{lemma}
\begin{proof} We consider four cases depending on whether $n\equiv 0,2\tmod{4}$ and $M\equiv 0,1\tmod{4}$.
In each of these cases, we will construct an $(n\times c)$ matrix $T$, where $\{2,3\}\ni c\equiv M\tmod{2}$, satisfying the following conditions:
\begin{enumerate}
  \item each column of $T$ is a permutation of $\Z_n$;
  \item $T$ has $\alpha=n-\beta$ rows each of which sums to $0$;
  \item $T$ has $\beta$ rows each of which sums to $\pm1$, or 
  $
    \begin{cases}
      \frac{n}{2}\pm1 & \text{if $n\equiv 0\tmod{4}$,}\\
      \frac{n}{2}\pm2 & \text{if $n\equiv 2\tmod{4}$.}
    \end{cases}
  $
\end{enumerate}
Note that $\frac{n}{2}\pm1$ is coprime to $n$ if and only if 
$n\equiv 0\tmod{4}$; therefore, $\frac{n}{2}\pm1$ has order $n$ in $\Z_n$. Similarly, 
$\frac{n}{2}\pm2$  has order $n$ in $\Z_n$ if and only if
$n\equiv 2\tmod{4}$. The assertion then follows by applying 
Theorem \ref{matrix} to $T$.

  We first consider the case where $n\equiv 2\pmod{4}$ and $M$ is even; thus, by assumption, we have that
  $\beta$ is even. If $n=2$, then $\beta=2$ (since, by assumption, $\beta>0$) and we set 
  $T=
  \left[
  \begin{array}{rr}
    0   & 1  \\
    1   & 0     
  \end{array}
  \right] 
  $.
 We now assume that $n\geq 6$. For $i\in\{2,4,6\}$ we first define the $6\times 2$ matrix $C_{i}$ as follows:
\[
C_2=
\left[
  \begin{array}{rr}
    -1        & 1     \\
    0         &  0    \\    
    1         & -1    \\  
    \frac{n}{2} & 2   \\
    2         & -2    \\    
    -2        &  \frac{n}{2}    \\       
  \end{array}
\right]\;\;\;
C_4=
\left[
  \begin{array}{rr}
    0        & 1     \\
    1        & -1    \\    
    -1       & 0     \\  
    \frac{n}{2} & 2  \\
    2        & -2    \\    
    -2       &  \frac{n}{2}    \\       
  \end{array}
\right] \;\;\;
C_6=
\left[
  \begin{array}{rr}
    0        & 1     \\
    2        & -1    \\    
    -1       & 0     \\  
    \frac{n}{2} & 2  \\
    1        & -2    \\    
    -2       &  \frac{n}{2}    \\       
  \end{array}
\right].
\]
Clearly, each column of $C_i$ uses up all integers in $[-2,2]\cup\{\frac{n}{2}\}$. Also, $i$ rows of
$C_i$ sum to $\pm1$ or $\frac{n}{2}\pm2$, which are all elements of order $n$ in $\Z_n$. Each of the remaining $6-i$ rows sums to $0$. Now, for each value of $\beta$, we define an $n\times 2$ matrix $T$ satisfying conditions $1-3$ as follows:
\[
\begin{tabular}{c|c}
$\beta=2$ & $4\leq\beta\equiv i\tmod{4}$ with $i\in\{4,6\}$\\ \hline
&\\
$
T=
\left[
  \begin{array}{c}
    \;\;\; A(3,\frac{n}{2}-1)             \\
   -A(3,\frac{n}{2}-1)             \\
    C_\beta                     
  \end{array}
\right] 
$
&
$
T=
\left[
  \begin{array}{c}
    \;\;\;A(\frac{\beta-i}{2}+3, \frac{n}{2}-1)             \\
   -A(\frac{\beta-i}{2}+3, \frac{n}{2}-1)             \\
    \;\;\;B(3,\frac{\beta-i}{2}+2)             \\    
   -B(3,\frac{\beta-i}{2}+2)             \\
    C_i                   
  \end{array}
\right] 
$ 
\end{tabular}
\]

  We now let $n\equiv 2\pmod{4}$ and $M$ be odd. Note that, by assumption, we have that
  $\beta>0$ is odd. If $n=2$, we set 
  $T=\left[
  \begin{array}{rrr}
    0  &  0 & 0     \\
    1  &  1 & 1
  \end{array}
  \right]$. We now assume that $n\geq 6$ and we note that by condition 2 we have that  
   $\beta\geq3$,  For $i\in\{3,5\}$ we define the $6\times 2$ matrix $C_{i}$ as follows:
\[
C_3=
\left[
  \begin{array}{rr}
    -1        & 2   \\  
    -2        &  \frac{n}{2}    \\        
    0         &  0    \\    
    1         & -1    \\  
    2         & -2    \\    
 \frac{n}{2}  & 1       
  \end{array}
\right]\;\;\;\;\;
C_5=
\left[
  \begin{array}{rr}
    -1        & 2  \\  
    0         &-1     \\    
    1         & 0     \\  
    -2        & \frac{n}{2}     \\    
    2         & -2     \\       
  \frac{n}{2} & 1        
  \end{array}
\right].
\]
Clearly, both columns of $C_i$ use up all integers in $[-2,2]\cup\{\frac{n}{2}\}$. 
Also, each of the first $i-1$ rows of $C_i$ sums to $\pm1$ or $\frac{n}{2}-2$, 
the last row of $C_i$ sums to $\frac{n}{2}+1$, and the remaining $6-i$ rows sum to $0$.
We now define an $n\times 2$ matrix $R$ according to the possible values of $\beta$:
\[
R=
\left[
  \begin{array}{c}
    \;\;\;A(\frac{\beta-i}{2}+3, \frac{n}{2}-1)             \\
   -A(\frac{\beta-i}{2}+3, \frac{n}{2}-1)             \\
    \;\;\;B(3,\frac{\beta-i}{2}+2)             \\    
   -B(3,\frac{\beta-i}{2}+2)             \\
    C_i                   
  \end{array}
\right]\;\;\;
\text{where $3\leq\beta\equiv i\tmod{4}$ with $i\in\{3,5\}$}.
\]
Clearly, each column of $R$ is a permutation of $\Z_n$. Further, $R$ has $\alpha$ rows whose sum is $0$, and $\beta-1$ rows each of which sums to $\pm1$ or $\frac{n}{2}\pm2$, whereas the last row sums to 
$\frac{n}{2}+1$.
To construct the requisite $(n\times 3)$ matrix $T$ satisfying conditions $1-3$, 
we consider a Skolem sequence $\{(a_i,b_i)\mid i\in[0,n/2-1]\}$ covering $[1,n-1]$ (which exists by 
Corollary \ref{skolem cor}) and replace each element $i\in \left[-\frac{n}{2}+1,\frac{n}{2}\right]$ in the first column of $R$ with the pair $(x_i,y_i)$ defined below:
\begin{equation}\label{sub}
  (x_i,y_i)=
  \begin{cases}
    (b_i, -a_i) & \text{if $i\in[0,\frac{n}{2}-1]$},  \\
    (a_{-i}, -b_{-i}) & \text{if $i\in-[1,\frac{n}{2}-1]$}, \\    
    (0,0) & \text{if $i=\frac{n}{2}$}.    
  \end{cases}
\end{equation}
It is not difficult to check that the new matrix $T$ satisfies conditions $1-3$. In fact, the first column (resp., second column) of $T$ uses up all integers in $[1,n]$ (resp., $-[1,n]$), therefore they are both permutations of 
$\Z_n$. We also point out that the above substitution preserves the sum of each row, except for the last row of $T$, which is $[0\;  0\;  1]$, and thus sums to $1$, and therefore yields a $C_{Mn}$-factor. 

Now, let $n\equiv 0\pmod{4}$; thus, by assumption, 
  $\beta>0$ is even. For $i\in\{0,2,4\}$ we define the $4\times 2$ matrix $C_{i}$ as follows:
\[
C_0=
\left[
  \begin{array}{rr}
    0           &  0    \\
    1           &  -1    \\    
    -1          &  1    \\  
   \frac{n}{2}  & \frac{n}{2} \\     
  \end{array}
\right]\;\;\;\;\;
C_2=
\left[
  \begin{array}{rr}
    0           &  0    \\
    1           & -1    \\      
    \frac{n}{2} &  1    \\    
   -1           & \frac{n}{2} \\     
  \end{array}
\right]\;\;\;\;\;
C_4=
\left[
  \begin{array}{rr}
    0           &  1    \\
    1           &  0    \\    
    \frac{n}{2} & -1    \\  
   -1           & \frac{n}{2} \\     
  \end{array}
\right].
\]
Clearly, both columns of $C_i$ use up all integers in $[-1,1]\cup\{\frac{n}{2}\}$. Also, $i$ rows of
$C_i$ sum to $1$ or $\frac{n}{2}\pm1$, whereas the remaining $4-i$ row sums to $0$.
If $M$ is even, we define an $n\times 2$ matrix $T$ satisfying conditions $1-3$ as follows:
\[
\begin{tabular}{c}
$2\leq\beta\equiv i\tmod{4}$ with $i\in\{2,4\}$\\ \hline
\\
$
T=
\left[
  \begin{array}{c}
    \;\;\;A(\frac{\beta-i}{2}+2, \frac{n}{2}-1)             \\
   -A(\frac{\beta-i}{2}+2, \frac{n}{2}-1)             \\
    \;\;\;B(2,\frac{\beta-i}{2}+1)             \\    
   -B(2,\frac{\beta-i}{2}+1)             \\
    C_i                   
  \end{array}
\right] 
$ 
\end{tabular}
\]
If $M$ is odd, to construct the required $(n\times 3)$ matrix  satisfying conditions $1-3$, 
we consider a Skolem sequence $\{(a_i,b_i)\mid i\in[0,n/2-1]\}$ of $[1,n-1]$
(which exists by  Corollary \ref{skolem cor})
 and replace each element $i$ in the first column of $T$ with the pair $(x_i,y_i)$ defined in equation $\eqref{sub}$.
It is not difficult to check that the new matrix
 satisfies conditions $1-3$ and this completes the proof.
\end{proof}

\begin{lemma}\label{C_M[n] n even beta=Mn/2+1}
Let $n\geq 2$ be even, $M\geq3$, and $0< \beta\leq n$.
Then $(n-\beta,\beta) \in \mathrm{HWP}(C_M[n];$ $M,Mn)$ whenever the following conditions simultaneously 
hold:
\begin{enumerate}
\item $\beta\equiv \frac{Mn}{2}+1 \tmod{2}$;
\item if $Mn\equiv 0 \tmod{4}$, then $\beta\neq1$;
\item if $Mn\equiv 2 \tmod{4}$ and $n>2$, then $\beta\neq2$.
\end{enumerate}
\end{lemma}
\begin{proof} We first consider the case where $Mn\equiv 0\tmod{4}$; hence, by assumption, we have that
$\beta$ is odd and $\beta\geq3$, thus $n\geq 4$. 
Let $T$ be the $(n-3)\times 2$ matrix with entries in 
$\Z_n\setminus\{\frac{n}{2}\pm1, \frac{n}{2}\}$ defined as follows:
\[
T=
\left[
  \begin{array}{c}
    A(-\frac{n}{2}+2, -\frac{n}{2}+1+\alpha)             \\
    B(-\frac{n}{2}+2+\alpha, \frac{n}{2}-2)             \\                   
  \end{array}
\right]. 
\]
Clearly, each column of $T$ is a permutation of 
$\Z_n\setminus\{\frac{n}{2}\pm1, \frac{n}{2}\}$,
each of the first $\alpha$ rows of $T$ sums to $0$, whereas each of the remaining $\beta-3$ 
sums to $\pm1$. 

We now construct an $(n-3)\times 3$ matrix $T'$ by modifying $T$ as follows.
By Corollary \ref{skolem cor}, there is a Skolem sequence $\{(a_i,b_i)\mid i\in[0,\frac{n}{2}-1]\}$ covering
$[-\frac{n}{2}+2,\frac{n}{2}-2]$. To construct $T$ we replace each element $i$ in the first column of $T'$ with the pair 
$(x_i,y_i)$ defined below:
\begin{equation}
  (x_i,y_i)=
  \begin{cases}
    (b_i, -a_i) & \text{if $i\in[0,\frac{n}{2}-2]$}, \\
    (a_{-i}, -b_{-i}) & \text{if $i\in-[1,\frac{n}{2}-2]$}. \\    
  \end{cases}
\end{equation}
It is not difficult to check that each of the first two columns of $T'$ uses up all integers in 
$[-\frac{n}{2}+2,\frac{n}{2}-2]$, therefore they are both permutations of 
$\Z_n\setminus\{\frac{n}{2}\pm1, \frac{n}{2}\}$. We also point out to the reader that the above substitution preserves the sum of each row. Therefore,  by applying 
Theorem \ref{matrix} to $T$ and $T'$, it follows that 
$(n-\beta, \beta-3)\in \mathrm{HWP}\left(C_M\left[\Z_n\setminus\left\{\frac{n}{2}\pm1, \frac{n}{2}\right\}\right]; 
M, Mn\right)$.
In view of Lemma \ref{6 regular}, 
$(0,3)\in \mathrm{HWP}\left(C_M\left[\left\{\frac{n}{2}\pm1, \frac{n}{2}\right\}\right]; M, Mn\right)$,
therefore
$(n-\beta, \beta)\in \mathrm{HWP}(C_M[n]; M, Mn)$.

We finally assume that $Mn\equiv 2\tmod{4}$; hence, by assumption,
$M$ is odd, $n\equiv2\tmod{4}$, and $\beta>0$ is even. 
If $n=2$ then
$(0,2)\in\mathrm{HWP}(C_M[2]; M, 2M)$ by Lemma \ref{hagg}. Therefore, we can assume that $n>2$, hence
 $\beta\geq4$ (condition 3).
First, let
$
T=
\left[
  \begin{array}{c}
    T_1     \\
    T_2     \\                
  \end{array}
\right]
$
be an $(n-2)\times 2$ matrix with entries in 
$\Z_n\setminus\{\frac{n}{2}-1,\frac{n}{2}\}$ where:
\[
T_1=
\left[
  \begin{array}{c}
    A(-\frac{n}{2}+3, -\frac{n}{2}+\alpha+2)             \\
    B(-\frac{n}{2}+\alpha+3, \frac{n}{2}-2)             \\                
  \end{array}
\right] \;\;\;\text{and}\;\;\;
T_2=
\left[
  \begin{array}{cc}
    -\frac{n}{2}+1 &  \frac{n}{2}-2              \\   
    -\frac{n}{2}+2 & -\frac{n}{2}+1             \\               
  \end{array}
\right].
\]
Note that each column of $T$ is a permutation of 
$\Z_n\setminus\{\frac{n}{2}-1, \frac{n}{2}\}$; also,
each of the first $\alpha$ rows of $T_1$ sums to $0$, whereas each of the following $\beta-4$ rows
sums to $\pm1$. 

We now construct an $(n-2)\times 3$ matrix $T'$ as follows. By Corollary \ref{skolem cor}
there is a Skolem sequence $\{(a_i,b_i)\mid i\in[0,\frac{n}{2}-2]\}$ covering
$[-\frac{n}{2}+2,\frac{n}{2}-2]$.  
As before, to construct $T'$ we replace each element of the second column of T, 
say $i\in[-\frac{n}{2}+1, \frac{n}{2}-2]$,  with the pair 
$(x_i,y_i)$ defined below:
\begin{equation}
  (x_i,y_i)=
  \begin{cases}
    (b_i, -a_i) & \text{if $i\in[0,\frac{n}{2}-2]$}, \\
    (a_i, -b_i) & \text{if $i\in-[1,\frac{n}{2}-2]$}, \\    
   (-\frac{n}{2}+1, -\frac{n}{2}+1) & \text{if $i=-\frac{n}{2}+1$}.
  \end{cases}
\end{equation}
It is not difficult to check that each of the columns of $T'$ uses up all integers in 
$[-\frac{n}{2}+1,\frac{n}{2}-2]$, that is, each of them is a permutation of 
$\Z_n\setminus\{\frac{n}{2}, \frac{n}{2}-1\}$. We also point out that the substitution $i\mapsto (x_i, y_i)$ preserves the sum of each row, except that the last row of $T'$ sums to $\frac{n}{2}+4$ which is coprime to $n$. 
Therefore,  by applying 
Theorem \ref{matrix} to $T$, it follows that 
$(n-\beta, \beta-2)\in \mathrm{HWP}(C_M[\Z_n\setminus\{\frac{n}{2}-1, \frac{n}{2}\}]; M, Mn)$.
By Lemma \ref{4reg}, 
$(0,2)\in \mathrm{HWP}(C_M[\{\frac{n}{2}-1, \frac{n}{2}\}]; M, Mn)$,
therefore
$(n-\beta, \beta)\in \mathrm{HWP}(C_M[n]; M, Mn)$.
\end{proof}

Lemmas  \ref{C_M[n] n odd} -- \ref{C_M[n] n even beta=Mn/2+1} clearly yield the following result.

\begin{theorem}
  \label{C_M[n]}
Let $n\geq 2$, $M\geq3$, and $0\leq y\leq n$.
Then $(n-y,y) \in \mathrm{HWP}(C_M[n];$ $M,Mn)$ except possibly when at least one of the following conditions  holds:
\begin{enumerate}
\item $y=1$ and $(n,(-1)^M)\neq (2,-1)$; 
\item $y=2<n\equiv2\tmod{4}$ and $M$ is odd;
\item $(y,n)\in\{(0,2),(0,6)\}$ and $M$ is odd.
\end{enumerate}
\end{theorem}

\section{Determining HWP($v; M, Mn$)}
\label{Main section}
In this section we prove the main result of this paper which concerns the existence of a solution to
HWP($K_v^*; M,N; \alpha, \beta$) when $M$ is a divisor of $N$.
Note that when $\alpha=0$ or $\beta=0$,
 this problem is equivalent to determining a $C_\ell$-factorization of $K_v^*$  and in this case a complete solution 
is provided by Theorem \ref{OP uniform}.

We denote by HW($G; M, N; \alpha, \beta)$ any solution to HWP($G; M, N; \alpha, \beta)$, that is,
any factorization of $G$ into $\alpha$ $C_M$-factors and $\beta$ $C_N$-factors.
We first prove the following lemma which provides sufficient conditions for the existence of an
HW($G; M, N; \alpha, \beta)$ for a given graph $G$.
%
%
%

\begin{lemma} \label{lemma:C_m[gn]-factorization}
Let $M,N,\alpha,\beta$ be positive integers with $M$ being a divisor of $N$ and $N>M\geq 3$. Also,
assume that $G$ has a factorization into $r\geq2$ $C_{M}[n]$-factors where $n=N/M$.
Then, $(\alpha, \beta)\in HWP(G; M, N)$ if and only if $\alpha+\beta=rn$,
except possibly when at least one of the following conditions holds:
\renewcommand{\theenumi}{(\roman{enumi})}
\begin{enumerate}
\item \label{general1} $\beta=1$;
\item \label{general2} $\beta=2<n\equiv2\tmod{4}$ and $M$ is odd;
\item \label{general3} $n=2$, $M$ is even, and $\beta$ is odd;
\item \label{general4} $n=2$, $M$ is odd, and $\beta<r$;
\item \label{general5} $n=6$, $M$ is odd, and $\beta<3r$.
\end{enumerate}
\end{lemma}
\begin{proof} Set $n=N/M$ and note that $n\geq2$ since $N>M$. By assumption, $G$ has a $C_{M}[n]$-factorization 
$\mathcal{\mathcal{G}}=\{G_1,$ $G_2, \ldots,$ $G_r\}$ with $r\geq 2$. It follows that 
$G$ is a regular graph of degree $2rn$. Now note that if $(\alpha, \beta)\in \mathrm{HWP}(G; M, N)$,
then $G$ has degree $2(\alpha+\beta)$, therefore $\alpha+\beta=rn$. 

We now show sufficiency; hence, we assume that $\alpha+\beta=rn$. 
We will proceed by applying Theorem \ref{C_M[n]} to factorize each 
of the $r$ $C_{M}[n]$-factors $G_i$ into an HW$(G_i; M, N; \alpha_i, \beta_i)$ where 
$\alpha=\sum_i \alpha_i$ and $\beta = \sum_i \beta_i$ for $i\in[1,r]$. Clearly,
this will result in an HW$(G; M, N; \alpha, \beta)$.

Set $\beta = xn + y$, with $0 \leq x < r$ and $0 \leq y < n$; note that by assumption $\beta>0$, and by exception 
$(i)$ we have that $\beta\neq 1$, hence 
$(x,y)\not\in\{(0,0), (0,1)\}$. We first assume that $n\not\in\{2,6\}$.
By taking into account exceptions $(ii)$, the following condition holds:
\begin{enumerate}
\item[$(a)$]  if $(x,y)=(0,2)$ $(i.e., \beta=2$) and $M$ is odd, then $n\not\equiv 2\tmod{4}$.
\end{enumerate}
We start with the case where $y\not\in\{1,2\}$ and  
apply Theorem \ref{C_M[n]} to 
fill $x$ $C_{M}[n]$-factors with an HW$(C_{M}[n];$ $M, Mn;  0, n)$,
     one $C_{M}[n]$-factor with an HW$(C_{M}[n];$ $M, Mn;$ $n-y,y)$,
and  the rest with an HW$(C_{M}[n];$ $M, Mn;$ $n, 0)$. 
If $(x,y)=(0,2)$, then in view of condition $(a)$ we can apply
Theorem \ref{C_M[n]} to 
fill $1$ $C_{M}[n]$-factor with an HW$(C_{M}[n];$ $M, Mn;  n-y, y)$
and  the rest with an HW$(C_{M}[n];$ $M, Mn;$ $n, 0)$. 
We finally consider the case where  $y\in\{1,2\}$ and $x>0$.
We again apply Theorem \ref{C_M[n]} to 
fill $x-1$ $C_{M}[n]$-factors with an HW$(C_{M}[n];$ $M, Mn;  0, n)$. If $n\geq 4$, we proceed by filling
     one $C_{M}[n]$-factor with an HW$(C_{M}[n];$ $M, Mn;$ $2,n-2)$ and
     one $C_{M}[n]$-factor with an HW$(C_{M}[n];$ $M, Mn;$ $n-y-2,y+2)$,     
If $n=3$ and $y=1$, then we proceed by filling two $C_{M}[n]$-factors with an HW$(C_{M}[n];$ $M, Mn;$ $1,2)$. 
If $n=3$ and $y=2$, then we fill
     one $C_{M}[n]$-factor with an HW$(C_{M}[n];$ $M, Mn;$ $0,3)$ and
     one $C_{M}[n]$-factor with an HW$(C_{M}[n];$ $M, Mn;$ $3-y,y)$,     
We fill the remaining $r-x-1$ $C_{M}[n]$-factors
with an HW$(C_{M}[n];$ $M, Mn;$ $n, 0)$. 

Now, we consider the case where $n\in\{2,6\}$ and $M$ is even. 
Note that when $n=2$, then $\beta$ is even 
(exception $(iii)$), that is, $y=0$. If $y\neq 1$, then
we apply Theorem \ref{C_M[n]} to 
fill $x$ $C_{M}[n]$-factors with an HW$(C_{M}[n];$ $M, Mn;  0, n)$, 
one $C_{M}[n]$-factor with an HW$(C_{M}[n];$ $M, Mn;$ $n-y,y)$,      
and  the rest with an HW$(C_{M}[n];$ $M, Mn;$ $n, 0)$. 
If $y = 1$, then $n=6$ and $x>0$ (since $(x,y)\neq (0,1)$). We 
apply again Theorem \ref{C_M[n]} to 
fill $x-1$ $C_{M}[n]$-factors with an HW$(C_{M}[n];$ $M, Mn;  0, n)$, 
one $C_{M}[n]$-factor with an HW$(C_{M}[n];$ $M, Mn;$ $1, n-1)$, 
one $C_{M}[n]$-factor with an HW$(C_{M}[n];$ $M, Mn;$ $n-2, 2)$, 
and  the rest with an HW$(C_{M}[n];$ $M, Mn;$ $n, 0)$. 

We finally assume that  $n\in\{2,6\}$ and $M$ is odd, and
set $\beta = x' r + y'$, with $0 \leq x' < n$ and $0 \leq y'< r$. 
In view of exceptions $(iv)-(v)$ we have that 
$x'\geq 1$  when $n=2$, and $x'\geq 3$  when $n=6$.
We can then apply Theorem \ref{C_M[n]} to 
fill $y'$ $C_{M}[n]$-factors with an HW$(C_{M}[n];$ $M, Mn;  n-x'-1, x'+1)$
and the remaining $(r-y')$ 
$C_{M}[n]$-factors with an HW$(C_{M}[n];$ $M, Mn;  n-x', x')$, and this completes the proof.
\end{proof}

We are now ready to prove the main result of this paper.\\

\noindent
\textbf{Theorem \ref{main}.}
\emph{
  Let $M,N,v, \alpha,\beta$ be positive integers such that $N>M\geq 3$ and $M$ is an odd divisor of $N$.
  Then, $(\alpha, \beta)\in \mathrm{HWP}(v; M, N)$ if and only if 
  $N\mid v$ and $\alpha+\beta=\lfloor\frac{v-1}{2}\rfloor$ except possibly when at least one of the following conditions holds:
\renewcommand{\theenumi}{(\roman{enumi})}  
  \begin{enumerate}
  \item $\beta=1$;
  \item $\beta=2$, $N\equiv 2M \tmod{4M}$;
  \item $N\in\{2M, 6M\}$;
  \item $v\in\{N, 2N, 4N\}$;
  \item $(M,v)=(3,6N)$. 
  \end{enumerate}
}
\begin{proof}
  We first note that by Theorem \ref{nec} if $(\alpha, \beta)\in {\rm HWP}(v; M, N)$, then  necessarily
$\alpha+\beta = \lfloor\frac{v-1}{2}\rfloor$, and both $M$ and $N$ are divisors of $v$.

We now show sufficiency; therefore, let $(v,M,N,\alpha, \beta)$ a quintuple which satisfies the assumptions of this theorem. Therefore, $v=Mns$ where $n=N/M$ and $s$ is a suitable positive integer. 
Also, in view of the possible exceptions $(i)-(v)$, we can assume that the following conditions simultaneously hold:
\begin{equation}\label{exceptions}
\begin{aligned}
  & \text{$\beta\neq 1$, $\beta\neq 2$ when $n\equiv 2\tmod{4}$, $n\not\in\{2,6\}$,} \\
  & \text{$s\not\in\{1,2,4\}$, and $(M,s)\neq (3,6)$.}
\end{aligned}
\end{equation}
We now set $w=Mn\frac{s}{t}$
where $t=s$ if $s$ is odd, otherwise $t=s/2$.
Note that in both cases we have $t\geq3$, since $s\not\in\{1,2,4\}$.

We factorize $K_{v}^*$ into 
$G_0=tK^*_{w}$  and $G_1=K_t[w]$.
 We start by applying Theorem~\ref{OP uniform} which guarantees the existence 
 of either a $C_{M}$- or a $C_{N}$-factorization of $G_0$ as we choose.
Therefore, this step will yield either  $\gamma$
$C_{M}$-factors or $\gamma$ $C_{N}$-factors decomposing $G_0$,
where 
$\gamma=\left\lfloor\frac{w-1}{2}\right\rfloor$.
More precisely, 
let $(\alpha_0, \beta_0)$ be the pair defined as follows:
\[(\alpha_0, \beta_0)=
\begin{cases}
  (\gamma,0) & \text{if $\beta< \gamma+3$,} \\
  (0,\gamma) & \text{if $\beta\geq \gamma+3$,} \\  
\end{cases}
\]
and apply Theorem~\ref{OP uniform} to fill 
$G_0$ with an $HW(G_0; M, N; \alpha_0,\beta_0)$. Since  $(M,s)\neq (3,6)$, by applying Corollary \ref{LiuGen} with $z=M\frac{s}{t}$ we obtain a $C_{M}[n]$-factorization of $K_t[w]$ containing at least three factors.
By taking into account Lemma 
\ref{lemma:C_m[gn]-factorization} and conditions \eqref{exceptions}, it follows that 
there exists an $\mathrm{HW}(G_1;$ $M,N;\alpha-\alpha_0,\beta-\beta_0)$ which we use to fill $G_1$ and this completes the proof.
\end{proof}

We point out that the above result has been proven in \cite{BDT2} in the case in which both $M$ and $N$ are odd,
but gives new results when $M$ is odd and $N$ is even.

The following corollary easily follows.

\begin{cor} Let $M\geq 3$ be an odd divisor of $N$. 
The necessary conditions for the solvability of  $\mathrm{HWP}(v;$ $M, N;$ $\alpha, \beta)$  are sufficient 
whenever $v>6N>36M$ and $\beta\geq 3$. 
\end{cor}


\section*{Acknowledgements}

The authors gratefully acknowledge support from the following sources.  A.C.\ Burgess and P.\ Danziger have received support from NSERC Discovery Grants RGPIN-435898-2013 and RGPIN-2016-04178, respectively. 
T. Traetta is a Marie Curie fellow of the Istituto Nazionale di Alta Matematica and gratefully acknowledges their support.

\end{document}